\documentclass[11pt]{amsart}

\usepackage{latexsym,rawfonts}
\usepackage{amsfonts,amssymb}
\usepackage{amsmath,amsthm}

\usepackage[plainpages=false]{hyperref}
\usepackage{graphicx}
\usepackage{color}

\newtheorem{theorem}{Theorem}[section]
\newtheorem{corollary}{Corollary}[section]
 
\newtheorem{lemma}{Lemma}[section]

\theoremstyle{definition} 
\newtheorem{definition}{Definition}[section]

\numberwithin{equation}{section}

\DeclareMathOperator{\VOL}{vol}
\DeclareMathOperator{\AREA}{area}

\DeclareMathOperator{\DIAG}{diag}

\newcommand{\R}{\mathbb{R}}
\newcommand{\uS}{\mathbb{S}^{n-1}}

\newcommand{\dd}{\mathop{}\!\mathrm{d}}

\newcommand{\set}[1]{\left\{#1\right\}}
\newcommand{\norm}[1]{\left\Vert#1\right\Vert}

\newcommand{\pd}{\partial}
\newcommand{\delbar}{\overline{\nabla}}

\newcommand{\A}{\mathcal{A}}

\newcommand{\MA}{Monge-Amp\`ere }

\begin{document}

\title{Non-uniqueness of solutions to the dual $L_p$-Minkowski problem}

\author{ Qi-Rui Li \quad Jiakun Liu \quad Jian Lu }

\address{Qi-Rui Li:
  School of Mathematical Sciences,
  Zhejiang University, Hangzhou 310027, China}
\email{qi-rui.li@zju.edu.cn}

\address{Jiakun Liu:
  School of Mathematics and Applied Statistics,
  University of Wollongong, Wollongong NSW 2522, Australia}
\email{jiakunl@uow.edu.au}

\address{Jian Lu:
  South China Research Center for Applied Mathematics and Interdisciplinary Studies, 
  South China Normal University, Guangzhou 510631, China; and
  \newline
  School of Mathematics and Applied Statistics,
  University of Wollongong, Wollongong NSW 2522, Australia}
\email{lj-tshu04@163.com}

\thanks{Research of Liu was supported by the Australian Research Council DP170100929; Research of Lu was supported by ARC DP170100929 and Natural Science Foundation of China (11871432)}

\date{}

\begin{abstract}
The dual $L_p$-Minkowski problem with $p<0<q$ is investigated in this
paper.
By proving a new existence result of solutions and constructing an example, we
obtain the non-uniqueness of solutions to this problem.
\end{abstract}



\maketitle

\baselineskip=16.4pt
\parskip=3pt

%% ------------------------------------------------------------

\section{Introduction}

In this paper, we consider the following \MA type equation:
\begin{equation} \label{pqMP}
  H^{1-p} |\delbar H|^{q-n} \det(\nabla^2H+HI) = f \quad \text{ on } \ \uS,
\end{equation}
where $H$ is the support function of a convex body $K=K_H$ in the Euclidean
space $\R^n$, $\nabla$ is the covariant derivative with respect to an
orthonormal frame on the unit sphere $\uS$, $I$ is the unit matrix of order $n-1$, $\delbar
H(x) = \nabla H(x) +H(x) x$ is the point on $\pd K$ whose unit outer normal
vector is $x\in\uS$, the indices $p\in\R$, $q\in\R$, and $f$ is a given positive
function on $\uS$.

Equation \eqref{pqMP} arises from the \emph{dual $L_p$-Minkowski problem}, when the given
measure is absolutely continuous.
The dual $L_p$-Minkowski problem was recently introduced by Lutwak, Yang and Zhang 
\cite{LYZ.Adv.329-2018.85},
which unifies the $L_p$-Minkowski problem and the dual Minkowski problem.
The $L_p$-Minkowski problem was introduced by Lutwak \cite{Lut.JDG.38-1993.131} in 1993 and has
been extensively studied over the last two decades
\cite{
  %BHZ.IMRNI.2016.1807, 
  BLYZ.JAMS.26-2013.831, 
  %BT.AiAM.87-2017.58,
  %CLZ.TAMS.371-2019.2623,
  %Che.Adv.201-2006.77,
  CW.Adv.205-2006.33,
  %DZ.Adv.230-2012.1209, 
  %HL.AiAM.50-2013.268,
  HLYZ.DCG.33-2005.699,
  JL.Adv.344-2019.262,
  JLW.JFA.274-2018.826, 
  JLZ.CVPDE.55-2016.41,
  %JWW.CVPDE.41-2011.535,
  %JW.DCDS.36-2016.785, 
  Lu.SCM.61-2018.511, 
  Lu.JDE.266-2019.4394, 
  LJ.DCDS.36-2016.971,
  LW.JDE.254-2013.983, 
  LYZ.JDG.56-2000.111,
  LYZ.TAMS.356-2004.4359,
  Sta.Adv.167-2002.160,
  %Zhu.Adv.262-2014.909,
  %Zhu.JFA.269-2015.1070,
  Zhu.JDG.101-2015.159}, see in particular, Schneider's book \cite{Schneider.2014} and the references
therein.
The dual Minkowski problem was first proposed in the recent groundbreaking work
\cite{HLYZ.Acta.216-2016.325}
and then followed by
\cite{
  BHP.JDG.109-2018.411,
  CL.Adv.333-2018.87,
  HP.Adv.323-2018.114,
  HJ.JFA, 
  JW.JDE.263-2017.3230,
  LSW.JEMS,
  Zha.CVPDE.56-2017.18,
  Zha.JDG.110-2018.543}.
For the dual $L_p$-Minkowski problem itself, various progress has been made after the
paper \cite{LYZ.Adv.329-2018.85} such as
\cite{
  BF.JDE.266-2019.7980,
  CHZ.MA.373-2019.953,
  CCL,
  CL,
  HZ.Adv.332-2018.57}.
However, there are still many unsolved problems in this emerging research area. 
In particular, very little is known about the uniqueness of solutions in the case when $q>p$, which may relate to different types of geometric inequalities \cite{LYZ.Adv.329-2018.85} and have some other applications. 

The aim of this paper is to establish some uniqueness and non-uniqueness results for the dual $L_p$-Minkowski problem, that is for solutions to the equation \eqref{pqMP}.
When $q=n$, equation \eqref{pqMP} is reduced to the $L_p$-Minkowski problem, of
which the uniqueness results can be found in e.g.
\cite{
  And.Invent.138-1999.151,
  BCD.Acta.219-2017.1,
  CHLL,
  CW.Adv.205-2006.33,
  %Cho.JDG.22-1985.117,
  DG.PJASAMS.70-1994.252,
  %Gag.DMJ.72-1993.441,
  HLX.Adv.281-2015.906,
  Lut.JDG.38-1993.131,
  Sta.Adv.180-2003.290}
and the non-uniqueness results can be found in e.g.
\cite{
  And.JAMS.16-2003.443,
  HLW.CVPDE.55-2016.117,
  JLW.Adv.281-2015.845,
  LiQR, 
  Yag.CVPDE.26-2006.49}.
For the general case of equation \eqref{pqMP},
it is already known that
the solution is unique when $q<p$
\cite{HZ.Adv.332-2018.57}
and is unique up to a constant when $q=p\neq0$
\cite{CL},
which can be obtained by the maximum principle.
The remaining case when $q>p$ is more complicated.
When $f\equiv1$, it was proved in \cite{CHZ.MA.373-2019.953} that
the even solution must be constant when $-n\leq p<q\leq\min\{n,p+n\}$,
and the evenness assumption was dropped in \cite{CL} when $1<p<q\leq n$.
On the other hand, for $f\equiv1$, there exists at least one non-constant
even solution when $n=2$, $q\geq6$, $p=0$ \cite{HJ.JFA},
which was later extended to either when $qp\geq0$, $q>p+2n$, or when $q>0>p$, $1+\frac{n}{p}<\frac{1}{p}+\frac{1}{q}<\frac{n}{q}-1$ and $q>p+2n$
\cite{CCL}.

In this paper, we consider the general function $f$ and obtain the following non-uniqueness of solutions to
equation \eqref{pqMP} for $p<0<q$.
\begin{theorem}\label{thm1}
  \textup{(1)} For any $0<q\leq1$ and $p<q-1$, there exists a positive
  function $f\in C^\infty(\uS)$ such that equation \eqref{pqMP} admits two
  different solutions.

  \textup{(2)} For any $q>1$ and $p<0$, there exists an almost
  everywhere positive function $f\in C^\infty(\uS)$ such that equation
  \eqref{pqMP} admits two different solutions. 
  If, in addition, $1+\frac{n}{p}<\frac{1}{p}+\frac{1}{q}<\frac{n}{q}-1$ holds, the function $f$ can be chosen to be
  positive on $\uS$. 
\end{theorem}

The proof of Theorem \ref{thm1} relies on a new existence result about equation \eqref{pqMP} in a class of symmetric convex bodies. 
In order to state this result, we first introduce some notations and terminologies. 
\begin{definition}
Denote $x=(x', x_n)\in\uS$, where $x'=(x_1, \cdots, x_{n-1})$.
A function $g:\uS\to\R$ is called \emph{rotationally symmetric} if it satisfies
\begin{equation*}
g(\Omega x',x_n)=g(x',x_n), \quad \forall \,x\in\uS \text{ and }\Omega\in \text{O}(n-1),
\end{equation*}
where $\text{O}(\cdot)$ denotes the orthogonal group.
A function $g:\uS\to\R$ is called \emph{even} if
\begin{equation*}
g(-x)=g(x), \quad \forall\,x\in\uS.
\end{equation*}
Correspondingly, when the support function of a convex body $K$ in $\R^n$ is
rotationally symmetric or even, we say $K$ is rotationally symmetric or even, respectively.
\end{definition}

\begin{definition}
When the origin is contained inside a convex body $K$, the radial function of $K$, denoted by
$\rho_K$, is defined as 
\begin{equation*}
\rho_K(u) :=\max\set{\lambda\geq0 : \lambda u\in K}, \quad u\in\uS.
\end{equation*}
Given a number $q\neq0$, the $q$-th dual volume of $K$ is defined by
\begin{equation}\label{qdvol}
    \widetilde{V}_q(K) = \frac{1}{n} \int_{\uS} \rho_{K}^q(u) \dd u.
\end{equation}
\end{definition}
We remark that the dual Minkowski problem is related to the first variation of the $q$-th dual volume \cite{HLYZ.Acta.216-2016.325}. See Lemma \ref{lem11} for its variational formula. 

\vskip5pt

We can now state our new existence result for equation \eqref{pqMP}, which will be used in the proof of Theorem \ref{thm1}. 
\begin{theorem}\label{thm2}
Assume that $p<0<q$, and $\alpha, \beta$ are two real numbers satisfying
\begin{align*}
  \alpha&>\max\set{1-n, 1-n+\tfrac{p(1-q)}{q}}, \\
  \beta&>\max\set{-1, -1+\tfrac{p(n-1-q)}{q}}.
\end{align*}
If $f$ is a non-negative, integrable,
rotationally symmetric, even function on $\uS$ and satisfies that
\begin{equation} \label{zero-cond}
  f(x)\leq C |x'|^\alpha \,|x_n|^\beta,  \qquad \forall \,x\in\uS
\end{equation}
for some positive constant $C$, and $\norm{f}_{L^1(\uS)}>0$,
then there exists a rotationally symmetric even solution to equation \eqref{pqMP}.
Moreover, the convex body $K$ determined by the solution satisfies the following estimates:
\begin{equation} \label{var-est}
  \widetilde{V}_q(K) \geq
  C_{n,p,q} \norm{f}_{L^1(\uS)}^{\frac{q}{q-p}},
\end{equation}
where $C_{n,p,q}$ is a positive constant depending only on $n$, $p$ and $q$.
\end{theorem}

Note that although Theorem \ref{thm2} is restricted to the rotationally
symmetric even case of equation \eqref{pqMP}, it is the first existence result
for all $p<0<q$.
If, in addition $p,q$ satisfy $1+\frac{n}{p}<\frac{1}{p}+\frac{1}{q}<\frac{n}{q}-1$, the existence in the even case was
previously obtained in \cite{CCL}.
%Our proof of Theorem \ref{thm2} is different and independent to that of \cite{CCL}.

\vskip5pt

We also remark that our method of constructing two different solutions in Theorem
\ref{thm1} is completely different from that of \cite{CCL, HJ.JFA}.
In \cite{CCL, HJ.JFA}, by the variational method a maximiser of some functional
provides one solution, and the constant function is obviously another solution.
Hence, the main effort in \cite{CCL, HJ.JFA} is to prove that the constant cannot be a maximiser. 
Here in our theorem, one solution is still from a maximiser of some
functional. However, the constant function is not a solution any more. 
So, the main difficulty is how to construct another solution that cannot be a maximiser of the functional.
In fact, our approach is inspired by \cite{JLW.Adv.281-2015.845}, where the
non-uniqueness of solutions to the $L_p$-Minkowski problem was obtained.
Since equation \eqref{pqMP} is more complicated than that of the $L_p$-Minkowski
problem, we have refined and improved the construction in \cite{JLW.Adv.281-2015.845} by
introducing an additional parameter, see \S\ref{s4}.
Surprisingly, this improvement can extend the result of \cite{JLW.Adv.281-2015.845} to all
$p\in(-\infty,0)$, as a special case of Theorem \ref{thm1} with $q=n$.
\begin{corollary}
For any $p\in(-\infty,0)$, there exists an almost everywhere positive function $f\in C^\infty(\mathbb{S}^{n-1})$ such that the $L_p$-Minkowski problem,
\begin{equation*} 
  \det(\nabla^2H+HI) = f H^{p-1} \quad \text{ on } \ \uS,
\end{equation*}
admits two different solutions. 
If $p\in(-n,0)$, the function $f$ can be chosen to be positive on $\uS$. 
\end{corollary}

This paper is organised as follows.
In \S\ref{s2}, we give some basic knowledge about convex bodies.
In \S\ref{s3}, we prove Theorem \ref{thm2} by a variational method.
In \S\ref{s4}, we prove Theorem \ref{thm1} based on the existence result in Theorem \ref{thm2} and a new construction of solutions.

\vskip5pt

\section{Preliminaries}\label{s2}

In this section we introduce some notations and preliminary results about convex bodies. 
The reader is referred to the newly expanded book \cite{Schneider.2014} of Schneider for a comprehensive introduction on the background.

A convex body $K$ is a compact convex subset of $\R^n$ with non-empty interior.
%We use $\INT K$ to denote the interior of $K$.
The support function of $K$ is given by
\begin{equation*}
  h_K(x) :=\max_{\xi\in K} \left\{\xi\cdot x\right\}, \quad x\in \uS,
\end{equation*}
where ``$\cdot$'' denotes the inner product in the Euclidean space $\R^n$.
It is well known that a convex body is uniquely determined by its support
function, and the convergence of a sequence of convex bodies is equivalent to
the uniform convergence of the corresponding support functions on $\uS$.
The Blaschke selection theorem says that every bounded sequence of convex bodies
has a subsequence that converges to a convex body.

Denote the set of positive continuous functions on $\uS$ by $C^+(\uS)$.
For $g\in C^+(\uS)$, the {\sl
Alexandrov body associated with $g$} is defined by
\begin{equation*}
  K := \bigcap_{x\in \uS} \set{\xi\in\R^n \, :\, \xi \cdot x\leq g(x)}.
\end{equation*}
One can see that $K$ is a bounded convex body and $0\in K$.
Note that
\begin{equation*}
  h_K(x) \leq g(x), \quad \forall\, x\in \uS.
\end{equation*}
Let $\widetilde{V}_q(g)$ be the $q$-th dual volume of the
Alexandrov body $K$ associated with $g$, defined in \eqref{qdvol}.
The following variational formula was obtained in \cite[Theorem 4.5]{HLYZ.Acta.216-2016.325}.

\begin{lemma} \label{lem11}
Let $\{G_t\}_{t\in(-\epsilon, \epsilon)}$ be a family of functions in $C^+(\uS)$, where $\epsilon>0$ is a small constant. 
  If there is a continuous function $g$ on $\uS$ such that
\begin{equation*}
  \lim_{t\to 0} \frac{G_t -G_0}{t}=g \quad \text{ uniformly on } \uS,
\end{equation*}
then
\begin{equation*}
  \lim_{t\to 0} \frac{\widetilde{V}_q(G_t) -\widetilde{V}_q(G_0)}{t}
  =\frac{q}{n} \int_{\uS} g(x) |\nu_K^{-1}(x)|^{q-n} \dd S_K(x),
\end{equation*}
where $K$ is the Alexandrov body associated with $G_0$,
$\nu_K^{-1}$ is the inverse Gauss map of $K$, and $\dd S_K$ is the surface area
measure of $K$.
\end{lemma}

\vskip5pt

\section{The existence of solutions}\label{s3}

In this section, we prove the existence result in Theorem \ref{thm2}.
Let $C^+_{re}(\uS)$ denote the set of rotationally symmetric, even and positive continuous functions on $\uS$.
Let $f\in C^+_{re}(\uS)$ be the function satisfying \eqref{zero-cond}, where the indices $p, q, \alpha, \beta$ are as in the hypotheses of Theorem \ref{thm2}. 
Consider the following maximising problem:
\begin{align}
  & \sup J[g] := \int_{\uS} f(x) g(x)^p \dd x,\quad\text{for } g\in C^+_{re}(\uS),  \label{MaxP} \\
  &\text{subject to } \quad  \widetilde{V}_q(K_g)=\kappa_n, \label{J}
\end{align}
where $\kappa_n$ is the volume of the unit ball in $\R^n$, and $K_g$ is the Alexandrov body associated with $g$.

\begin{lemma}\label{lem06}
  Under the assumptions of Theorem \ref{thm2}, the maximising problem
  \eqref{MaxP} has a solution $h$. In addition, the solution $h$ is the support
  function of $K_h$. 
\end{lemma}

Assuming this for the moment, we can then prove that a multiple of $h$ is a rotationally symmetric even solution to equation \eqref{pqMP}.

\begin{proof}[\textbf{Proof of Theorem \ref{thm2}}]
Let $h$ be the solution obtained in Lemma \ref{lem06}. 
For any given rotationally symmetric and even $\eta\in C(\uS)$, let
\begin{equation*}
  \phi_t=h+t\eta \quad \text{ for } t\geq0.
\end{equation*}
Since $h\in C^+_{re}(\uS)$,  for $t$ sufficiently small $\phi_t\in C^+_{re}(\uS)$ as well.
By Lemma \ref{lem11}, we have
\begin{equation} \label{eq:51}
  \lim_{t\to0^+} \frac{\widetilde{V}_q(K_{\phi_t}) -\widetilde{V}_q(K_h)}{t}
  =\frac{q}{n} \int_{\uS} \eta(x) |\nu_{K_h}^{-1}(x)|^{q-n} dS_{K_h}(x).
\end{equation}

Let $g_t(x)=\lambda(t)\phi_t(x)$, where
\begin{equation*}
\lambda(t)=\widetilde{V}_q(K_{\phi_t})^{-1/q}\kappa_n^{1/q}.
\end{equation*}
Then $g_t\in C^+_{re}(\uS)$, and $\widetilde{V}_q(K_{g_t})=\kappa_n$. 
Note that $g_0(x)=h(x)$, and
\begin{equation}\label{eq:52}
  \lim_{t\to0^+}\frac{g_t(x)-g_0(x)}{t}=\eta(x)+\lambda'(0)h(x) \quad\text{uniformly on }\uS.
\end{equation}
Also by \eqref{eq:51},
\begin{equation} \label{eq:53}
  \begin{split}
    \lambda'(0)
    &= -\frac{\kappa_n^{1/q}}{q} \widetilde{V}_q(K_{\phi_t})^{-\frac{1}{q}-1}
    \frac{\dd \widetilde{V}_q(K_{\phi_t})}{\dd t} \Big|_{t=0} \\
    &= -\frac{1}{n \kappa_n} \int_{\uS} \eta(x) |\nu_{K_h}^{-1}(x)|^{q-n} dS_{K_h}(x).
  \end{split}
\end{equation}

Let
\begin{equation*}
J(t):=J[g_t].
\end{equation*}
Note that $J(0)=J[h]$. Since $h$ is a maximister of \eqref{MaxP}, we have
\begin{equation*}
J(t)\leq J(0)
\end{equation*}
for all small $t\geq0$. Thus
\begin{equation*}
  \lim_{t_k\to0^+} \frac{J(t_k)-J(0)}{t_k}\leq0
\end{equation*}
for any convergent subsequence $\set{t_k}$.
Recalling \eqref{eq:52}, we see the above inequality can be simplified as
\begin{equation} \label{eq:55}
  \int_{\uS} p h(x)^{p-1} [\eta(x)+\lambda'(0)h(x)] f(x) \dd x \leq0.
\end{equation}
Since $p<0$, by \eqref{eq:53} we obtain
\begin{equation*}
\int_{\uS} \eta(x) h(x)^{p-1} f(x) \dd x -c\int_{\uS} \eta(x) |\nu_{K_h}^{-1}(x)|^{q-n} dS_{K_h}(x) \geq 0,
\end{equation*}
where
\begin{equation} \label{eq:57}
c=\frac{1}{n\kappa_n} \int_{\uS} fh^p.
\end{equation}
Replacing $\eta$ by $-\eta$, we see that
\begin{equation*}
\int_{\uS} \eta(x) h(x)^{p-1} f(x) \dd x -c\int_{\uS} \eta(x) |\nu_{K_h}^{-1}(x)|^{q-n} dS_{K_h}(x) = 0
\end{equation*}
for all rotationally symmetric and even $\eta\in C(\uS)$.
Note that $h$, $K_h$ and $f$ are rotationally symmetric and even, we obtain
\begin{equation*}
 h(x)^{p-1} f(x) \dd x =c |\nu_{K_h}^{-1}(x)|^{q-n} dS_{K_h}(x),
\end{equation*}
namely $h$ is a generalised solution to the equation
\begin{equation}\label{eq:56}
  h^{p-1}f=c\,|\delbar h|^{q-n}\det(\nabla^2h+hI) \quad \text{on }\uS.
\end{equation}

Let
\begin{equation*}
  H=c^{\frac{1}{q-p}}h,
\end{equation*}
then $H$ is a generalised solution to equation \eqref{pqMP}, namely 
\begin{equation*}
  H^{1-p} |\delbar H|^{q-n} \det(\nabla^2H+HI) = f \quad \text{on } \uS.
\end{equation*}
Now it remains to verify inequality \eqref{var-est}.
In fact, noting that
\begin{equation} \label{eq:54}
  \widetilde{V}_q(K_H)
  =c^{\frac{q}{q-p}}\widetilde{V}_q(K_h)
  =c^{\frac{q}{q-p}}\kappa_n,
\end{equation}
we need to estimate $c$.
% \begin{equation*}
%   \int_{\uS} fh^p =c\int_{\uS} h|\delbar h|^{q-n}\det(\nabla^2h+hI) =c\int_{\uS} \rho_{K_h}^q.
% \end{equation*}
Recalling \eqref{eq:57}, one has
\begin{equation*}
  \begin{split}
  c
  &=\frac{1}{n\kappa_n} J[h] \\
  &\geq \frac{1}{n\kappa_n} J[1] \\
  &= \frac{1}{n\kappa_n} \norm{f}_{L^1(\uS)}.
  \end{split}
\end{equation*}
%We remark that by the assumption $\norm{f}_{L^1(\uS)}>0$ of Theorem \ref{thm2}, the constant $c>0$. 
%Since $K_H$ is rotationally symmetric and even, when $f$ is smooth one can see that the generalised solution is also a classical solution by the regularity theory of Monge-Amp\`ere equations.  
The above estimate together with \eqref{eq:54} implies the estimate
\eqref{var-est}.
The proof of Theorem \ref{thm2} is completed.
\end{proof}

Therefore, it suffices to prove Lemma \ref{lem06}, namely the existence of maximiser $h$ of \eqref{MaxP}--\eqref{J}, which was crucially used in the above proof of Theorem \ref{thm2}.

\begin{proof}[Proof of Lemma \ref{lem06}]
For any $g\in C^+_{re}(\uS)$ satisfying $\widetilde{V}_q(K_g)=\kappa_n$, let $E_g$ be
the minimum ellipsoid of $K_g$. 
One has
\begin{equation*}
\frac{1}{n} E_g \subset K_g \subset E_g,
\end{equation*}
which implies that
\begin{equation} \label{eq:48}
  \begin{aligned}
    \frac{1}{n} h_{E_g} &\leq h_{K_g} \leq h_{E_g} \quad \text{on } \uS, \\
    \frac{1}{n} \rho_{E_g} &\leq \rho_{K_g} \leq \rho_{E_g} \quad \text{on } \uS,
  \end{aligned}
\end{equation}
where $h_\bullet$ and $\rho_\bullet$ are the corresponding support function and radial
function, respectively. 
Since $K_{g}$ is rotationally symmetric and even, $E_g$ is
also rotationally symmetric and even.
Therefore, the centre of $E_g$ is at the origin, and there exists a unique
rotationally symmetric matrix $A_g$ of the form 
\begin{equation} \label{A}
  A=\begin{pmatrix}
    ra^{\frac{1}{n}}&&& \\
    &\ddots&& \\
    &&ra^{\frac{1}{n}}& \\
    &&&ra^{\frac{1-n}{n}}
  \end{pmatrix},
  \quad \text{where }r>0, \ a>0\text{ are constants,}
\end{equation}
such that
\begin{equation} \label{eq:49}
  \begin{aligned}
    h_{E_g}(x)&=|A_gx|  &&\text{on } \uS, \\
    \rho_{E_g}(u)&=|A_g^{-1}u|^{-1}  &&\text{on } \uS.
  \end{aligned}
\end{equation}
By computation we have
\begin{equation*}
  \begin{split}
    \int_{\uS} |A_g^{-1}x|^{-q}
    &= \int_{\uS} \rho_{E_g}(u)^q \dd u \\
    &\leq n^q \int_{\uS} \rho_{K_g}(u)^q \dd u \\
    &= n^{q+1} \widetilde{V}_q(K_g) \\
    &= n^{q+1} \kappa_n,
  \end{split}
\end{equation*}
and
\begin{equation*}
  \begin{split}
    \int_{\uS} |A_g^{-1}x|^{-q}
    &= \int_{\uS} \rho_{E_g}(u)^q \dd u \\
    &\geq \int_{\uS} \rho_{K_g}(u)^q \dd u \\
    &= n \widetilde{V}_q(K_g) \\
    &= n \kappa_n.
  \end{split}
\end{equation*}
If we set $\Lambda=\max\set{n^{q+1}\kappa_n, (n\kappa_n)^{-1}}$, then
\begin{equation}\label{newb}
\Lambda^{-1}\leq \int_{\uS} |A_g^{-1}x|^{-q} \leq \Lambda.
\end{equation}

%Recalling the definition of $\mathcal{A}$ given in \eqref{eq:5}, we have $A_g\in\mathcal{A}$.

From the assumption \eqref{zero-cond}, and noting that
\begin{equation*}
  g(x)\geq h_{K_g}(x)\geq \frac{1}{n}h_{E_g}(x)
  =\frac{1}{n} |A_gx|
  \quad \text{on }\uS,
\end{equation*}
since $p<0$, one can obtain 
\begin{equation} \label{eq:4}
  \begin{split}
    J[g]
    &=\int_{\uS} f(x) g(x)^p \dd x \\
    &\leq n^{-p} \int_{\uS} f(x) |A_gx|^p \dd x \\
    &\leq C n^{-p} \int_{\uS} |x'|^\alpha \,|x_n|^\beta \,|A_gx|^p \dd x \\
    &=: C n^{-p} F(A_g),
  \end{split}
\end{equation}
where $C$ is the constant given in Theorem \ref{thm2}.
Thanks to \eqref{newb}, we \emph{claim} that
\begin{equation}\tag{{\bf \emph{Claim 1}}}\label{c1}
F(A_g) =\int_{\uS} |x'|^\alpha \,|x_n|^\beta \,|A_gx|^p \dd x \leq C,
\end{equation}
is bounded by a constant $C=C(n,p,q,\alpha,\beta)$, for all $g\in C^+_{re}(\uS)$ satisfying \eqref{J}, which then implies that
\begin{equation*}
  M :=
  \sup \set{ J[g] \,:\, g\in C^+_{re}(\uS), \  \widetilde{V}_q(K_g)=\kappa_n }
  < +\infty.
\end{equation*}

Assuming this at the moment, let $\set{g_k}\subset C^+_{re}(\uS)$ with $\widetilde{V}_q(K_{g_k})=\kappa_n$ be
a maximising sequence.
Denote $h_k=h_{K_{g_k}}$, namely $h_k$ is the support function of $K_{g_k}$.
Then
\begin{equation*}
h_k(x)\leq g_k(x), \quad \forall x\in\uS,
\end{equation*}
and $K_{h_k}=K_{g_k}$. Since $K_{g_k}$ is rotationally symmetric and even,
$h_k\in C^+_{re}(\uS)$.
Since $p<0$, one has $J[h_k]\geq J[g_k]$, which implies that
\begin{equation*}
\lim_{k\to+\infty} J[h_k]=\lim_{k\to+\infty} J[g_k]=M.
\end{equation*}
Namely, $\set{h_k}$ is also a maximising sequence of \eqref{MaxP}.

In order to show that $\set{h_k}$ has uniform positive upper and lower bounds on $\uS$,
we may suppose to the contrary that
\begin{equation} \label{eq:50}
  \lim_{k\to+\infty} \max_{x\in\uS} h_k(x)=+\infty
  \quad\text{or}\quad
  \lim_{k\to+\infty} \min_{x\in\uS} h_k(x)=0.
\end{equation}
For each $h_k$, recalling the construction at the beginning of the proof, there
exists a unique matrix $A_{h_k}$ of form \eqref{A}, satisfying \eqref{newb}.
Denote the corresponding $r$ and $a$ by $r_k$ and $a_k$ respectively.
By virtue of \eqref{eq:48} and \eqref{eq:49}, we have
\begin{equation*}
\frac{1}{n} |A_{h_k}x|\leq h_k(x)\leq |A_{h_k}x|, \quad \forall x\in\uS,
\end{equation*}
which together with \eqref{eq:50} implies that
\begin{equation*}
  \lim_{k\to+\infty} \max_{x\in\uS} |A_{h_k}x| =+\infty
  \quad\text{or}\quad
  \lim_{k\to+\infty} \min_{x\in\uS} |A_{h_k}x| =0.
\end{equation*}
By \eqref{A}, this means that
\begin{equation*}
  r_ka_k^{\frac{1}{n}}\to+\infty
  \quad\text{or}\quad
  r_ka_k^{\frac{1-n}{n}}\to+\infty 
  \quad\text{or}\quad
  r_ka_k^{\frac{1}{n}}\to0
  \quad\text{or}\quad
  r_ka_k^{\frac{1-n}{n}}\to0
\end{equation*}
as $k\to+\infty$.
Hence, one of the four cases $r_k\to+\infty$, $r_k\to0$, $a_k\to+\infty$, $a_k\to0$ must
occur. 

If one of the above cases occurs, we next \emph{claim} at the moment that
\begin{equation}\tag{{\bf \emph{Claim 2}}}\label{c2}
F(A_{h_k})\to0,\quad \text{ as } k\to+\infty.
\end{equation}
Recalling \eqref{eq:4}, we thus obtain
\begin{equation*}
\lim_{k\to+\infty} J[h_k]=0,
\end{equation*}
namely $M=0$.
However, choosing $g\equiv1$, one can see that
\begin{equation} \label{eq:58}
M\geq J[1] =\norm{f}_{L^1(\uS)} >0,
\end{equation}
which is a contradiction. 
Therefore, $\set{h_k}$ has uniform positive upper and lower
bounds. 

By the Blaschke selection theorem, there is a subsequence of $\set{h_k}$ that
uniformly converges to some support function $h$ on $\uS$. Correspondingly,
$K_{h_k}$ converges to $K_h$, which is the convex body determined by $h$.
Obviously, $h$ is rotationally symmetric and even on $\uS$, satisfying $h>0$,
$\widetilde{V}_q(K_h)=\kappa_n$, and
\begin{equation*}
J[h]=\lim_{k\to+\infty}J[h_k]=M.
\end{equation*}
Hence, we see that $h$ is a solution to the maximising problem \eqref{MaxP}--\eqref{J}, and
$h$ is the support function of $K_h$.
\end{proof}

Last, we prove the two claims used in the above proof. 
Let $\Lambda$ be the same constant as in \eqref{newb}.
Define
\begin{equation} \label{eq:5}
  \A :=\set{A \text{ has the form of \eqref{A}} \,:\,
    \Lambda^{-1} \leq \int_{\uS} |A^{-1}x|^{-q} \leq \Lambda}.
\end{equation}
Both {\bf Claims 1\&2} are proved in the following lemma. 

We need to estimate the following integral:
\begin{lemma}\label{lem01}
Assume the indices $p$, $q$, $\alpha$, $\beta$ satisfy the assumptions of Theorem \ref{thm2}. 
Let
\begin{equation} \label{FA}
  F(A) :=\int_{\uS} |x'|^\alpha \,|x_n|^\beta \,|Ax|^p \dd x.
\end{equation}

Then, there exists a positive constant $C$ only depending on $n$, $\Lambda$, $p$, $q$,
$\alpha$ and $\beta$, such that
\begin{equation} \label{eq:6}
  \sup_{A\in\A} F(A) \leq C.
\end{equation}
Moreover, for any sequence $\set{A_k}\subset\A$ with corresponding $r_k$ and
$a_k$, if one of the four cases $r_k\to\infty$, $r_k\to0$, $a_k\to\infty$,
$a_k\to0$ occurs, then one has $F(A_k)\to0$. 
\end{lemma}

\begin{proof}
The set $\A$ can be divided into the following three disjoint subsets:
 \begin{equation*}
\A=\A_1\cup\A_2\cup\A_3,
\end{equation*}
where
\begin{align*}
  \A_1&:=\A\cap\set{A : a>3}, \\
  \A_2&:=\A\cap\set{A : 1/3\leq a\leq 3}, \\
  \A_3&:=\A\cap\set{A : a<1/3}.
\end{align*}
To prove this lemma, it suffices to prove it on each subset $\A_i$, for $i=1,2,3$.
Note that if not explicitly stated, the positive constants $C$ and $\tilde{C}$ in the following proof only depend on $n$, $\Lambda$, $p$, $q$,
$\alpha$ and $\beta$, but may vary from line to line.

\textbf{Case $\A_1$.}
We first claim that there exist two positive constants $C$ and $\tilde{C}$ such that for any
$A\in\A_1$, one has
\begin{equation} \label{eq:11}
  \begin{aligned} 
    Ca^{1-\frac{1}{n}} &\leq r \leq \tilde{C}a^{1-\frac{1}{n}}, && \text{ if } q<1, \\
    Ca^{1-\frac{1}{n}}(\log a)^{-1} &\leq r \leq \tilde{C}a^{1-\frac{1}{n}}(\log a)^{-1}, && \text{ if } q=1, \\
    Ca^{\frac{1}{q}-\frac{1}{n}} &\leq r \leq \tilde{C}a^{\frac{1}{q}-\frac{1}{n}}, && \text{ if } q>1.
  \end{aligned}
\end{equation}

In fact, from \eqref{A}
\begin{equation*}
A^{-1}=\begin{pmatrix}
    r^{-1}a^{-\frac{1}{n}}&&& \\
    &\ddots&& \\
    &&r^{-1}a^{-\frac{1}{n}}& \\
    &&&r^{-1}a^{\frac{n-1}{n}}
  \end{pmatrix}, 
\end{equation*}
thus
\begin{equation*}
|A^{-1}x|=r^{-1}a^{1-\frac{1}{n}}\sqrt{a^{-2}|x'|^2+x_n^2}.
\end{equation*}
Computing in the spherical coordinates, we then have
\begin{equation} \label{eq:9}
  \begin{split}
    \int_{\uS} |A^{-1}x|^{-q}
    &= r^q a^{\frac{q}{n}-q} \int_{\uS} (a^{-2}|x'|^2+x_n^2)^{-\frac{q}{2}} \dd x \\
    &= 2r^q a^{\frac{q}{n}-q} \,\omega_{n-2} \int_0^{\frac{\pi}{2}}
    (a^{-2}\sin^2\theta+\cos^2\theta)^{-\frac{q}{2}} \sin^{n-2}\theta \dd\theta,
  \end{split}
\end{equation}
where $\omega_{n-2}$ denotes the surface area of the unit ball in $\R^{n-1}$.
Since $a>3$, we have
\begin{equation} \label{eq:8}
  \begin{split}
    \int_{\uS} |A^{-1}x|^{-q}
    &\geq C r^q a^{\frac{q}{n}-q} \Bigl[  \int_0^{\frac{\pi}{4}}  \sin^{n-2}\theta \dd\theta 
    +\int_{\frac{\pi}{4}}^{\frac{\pi}{2}} \bigl(a^{-2}+(\tfrac{\pi}{2}-\theta)^2\bigr)^{-\frac{q}{2}} \dd\theta \Bigr] \\
    &\geq C r^q a^{\frac{q}{n}-q} \Bigl[ 1
    +\int_0^{\frac{\pi}{4}} \bigl(a^{-2}+t^2\bigr)^{-\frac{q}{2}} \dd t \Bigr] \\
    &\geq C r^q a^{\frac{q}{n}-q} \Bigl[ 1 +\int_0^{\frac{\pi}{4}} \bigl(a^{-1}+t\bigr)^{-q} \dd t \Bigr].
  \end{split}
\end{equation}

Since
\begin{equation*}
\int_0^{\frac{\pi}{4}} \bigl(a^{-1}+t\bigr)^{-q} \dd t
=\begin{cases}
  \frac{1}{1-q} \bigl[ (a^{-1}+\tfrac{\pi}{4})^{1-q} -a^{q-1}  \bigr], & \text{ if } q<1, \\
  \log(1+\tfrac{\pi}{4}a), & \text{ if } q=1, \\
  \frac{1}{q-1} \bigl[ a^{q-1} -(a^{-1}+\tfrac{\pi}{4})^{1-q} \bigr], & \text{ if } q>1,
\end{cases}
\end{equation*}
we obtain from \eqref{eq:8} that
\begin{equation} \label{eq:10}
  \int_{\uS} |A^{-1}x|^{-q}
  \geq C r^qa^{\frac{q}{n}-q}
  \begin{cases}
    1, & \text{ if } q<1, \\
    \log a, & \text{ if } q=1, \\
    a^{q-1}, & \text{ if } q>1.
  \end{cases}
\end{equation}

On the other hand, from \eqref{eq:9}, we also have
\begin{equation*}
    \int_{\uS} |A^{-1}x|^{-q}
    \leq 
    \tilde{C} r^q a^{\frac{q}{n}-q} \Bigl[  \int_0^{\frac{\pi}{4}}  \sin^{n-2}\theta \dd\theta 
    +\int_{\frac{\pi}{4}}^{\frac{\pi}{2}} \bigl(a^{-2}+(\tfrac{\pi}{2}-\theta)^2\bigr)^{-\frac{q}{2}} \dd\theta \Bigr].
\end{equation*}
Similarly as in \eqref{eq:8}--\eqref{eq:10}, one can obtain
\begin{equation} \label{eq:12}
  \int_{\uS} |A^{-1}x|^{-q}
  \leq \tilde{C} r^qa^{\frac{q}{n}-q}
  \begin{cases}
    1, & \text{ if } q<1, \\
    \log a, & \text{ if } q=1, \\
    a^{q-1}, & \text{ if } q>1.
  \end{cases}
\end{equation}

By combining \eqref{eq:10}, \eqref{eq:12} and the definition of $\A$ (see
\eqref{eq:5}), we have
\begin{equation*}
  \begin{aligned} 
    \tilde{C}^{-1}\Lambda^{-1} &\leq r^qa^{\frac{q}{n}-q} \leq C^{-1}\Lambda,  && \text{ if } q<1, \\
    \tilde{C}^{-1}\Lambda^{-1} &\leq ra^{\frac{1}{n}-1}\log a \leq C^{-1}\Lambda,  && \text{ if } q=1, \\
    \tilde{C}^{-1}\Lambda^{-1} &\leq r^qa^{\frac{q}{n}-1} \leq C^{-1}\Lambda,  && \text{ if } q>1,
  \end{aligned}
\end{equation*}
which implies the claim \eqref{eq:11}.

Now, we estimate $F(A)$ given in \eqref{FA}. Noting that
\begin{equation*}
  |Ax|=ra^{\frac{1}{n}}\sqrt{|x'|^2+a^{-2}x_n^2},
\end{equation*}
we compute in the spherical coordinates as follows:
\begin{equation} \label{eq:7}
  \begin{split}
    F(A)
    &= r^p a^{\frac{p}{n}} \int_{\uS} |x'|^\alpha \,|x_n|^\beta (|x'|^2+a^{-2}x_n^2)^{\frac{p}{2}} \dd x \\
    &= 2r^p a^{\frac{p}{n}} \omega_{n-2} \int_0^{\frac{\pi}{2}} \sin^\alpha\!\theta \cos^\beta\!\theta
    \,(\sin^2\theta+a^{-2}\cos^2\theta)^{\frac{p}{2}} \sin^{n-2}\theta \dd\theta \\
    &\leq C r^p a^{\frac{p}{n}} \int_0^{\frac{\pi}{4}} \theta^{\alpha+n-2} (\theta^2+a^{-2})^{\frac{p}{2}}  \dd\theta \\
    &\hskip1.1em + C r^p a^{\frac{p}{n}} \int_{\frac{\pi}{4}}^{\frac{\pi}{2}} (\tfrac{\pi}{2}-\theta)^\beta
    \bigl(1+a^{-2}(\tfrac{\pi}{2}-\theta)^2\bigr)^{\frac{p}{2}} \dd\theta.
  \end{split}
\end{equation}
Recalling that $a>3$, we can further estimate these integrals:
\begin{equation} \label{eq:13}
  \begin{split}
    \int_0^{\frac{\pi}{4}} (\theta^2+a^{-2})^{\frac{p}{2}} \theta^{\alpha+n-2} \dd\theta
    &\leq \int_0^{\frac{1}{a}} a^{-p} \theta^{\alpha+n-2} \dd\theta
    +\int_{\frac{1}{a}}^{\frac{\pi}{4}} \theta^p \,\theta^{\alpha+n-2} \dd\theta \\
    &= \frac{a^{-p-\alpha-n+1}}{\alpha+n-1} 
    +\int_{\frac{1}{a}}^{\frac{\pi}{4}} \theta^{p+\alpha+n-2} \dd\theta \\
    &\leq C \begin{cases}
      1, & \text{ if } p+\alpha+n-1>0, \\
      \log a, & \text{ if } p+\alpha+n-1=0, \\
      a^{-p-\alpha-n+1}, & \text{ if } p+\alpha+n-1<0,
    \end{cases}
  \end{split}
\end{equation}
and
\begin{equation} \label{eq:14}
  \begin{split}
    \int_{\frac{\pi}{4}}^{\frac{\pi}{2}} (\tfrac{\pi}{2}-\theta)^\beta
    \bigl(1+a^{-2}(\tfrac{\pi}{2}-\theta)^2\bigr)^{\frac{p}{2}} \dd\theta
    &= \int_0^{\frac{\pi}{4}} t^\beta \bigl(1+a^{-2}t^2\bigr)^{\frac{p}{2}} \dd t \\
    &\leq C \int_0^{\frac{\pi}{4}} t^\beta \dd t \\
    &\leq C.
  \end{split}
\end{equation}
Inserting \eqref{eq:13} and \eqref{eq:14} into \eqref{eq:7}, we obtain
\begin{equation} \label{eq:15}
  F(A)\leq Cr^pa^{\frac{p}{n}}
  \begin{cases}
    1, & \text{ if } p+\alpha+n-1>0, \\
    \log a, & \text{ if } p+\alpha+n-1=0, \\
    a^{-p-\alpha-n+1}, & \text{ if } p+\alpha+n-1<0.
  \end{cases}
\end{equation}

Recall that any $A\in\A_1$ must satisfy \eqref{eq:11}. The estimate given by
\eqref{eq:15} can be simplified as follows.
When $q<1$, one has
\begin{equation} \label{eq:16}
  F(A)\leq C
  \begin{cases}
    a^p, & \text{ if } p+\alpha+n-1>0, \\
    a^p\log a, & \text{ if } p+\alpha+n-1=0, \\
    a^{-\alpha-n+1}, & \text{ if } p+\alpha+n-1<0.
  \end{cases}
\end{equation}
When $q=1$, one has
\begin{equation} \label{eq:17}
  F(A)\leq C
  \begin{cases}
    a^p(\log a)^{-p}, & \text{ if } p+\alpha+n-1>0, \\
    a^p(\log a)^{1-p}, & \text{ if } p+\alpha+n-1=0, \\
    a^{-\alpha-n+1}(\log a)^{-p}, & \text{ if } p+\alpha+n-1<0.
  \end{cases}
\end{equation}
And when $q>1$, one has
\begin{equation} \label{eq:18}
  F(A)\leq C
  \begin{cases}
    a^{p/q}, & \text{ if } p+\alpha+n-1>0, \\
    a^{p/q}\log a, & \text{ if } p+\alpha+n-1=0, \\
    a^{p/q-p-\alpha-n+1}, & \text{ if } p+\alpha+n-1<0.
  \end{cases}
\end{equation}
By the assumptions on $p$, $q$ and $\alpha$, we see that the power of $a$ in
each case of \eqref{eq:16}, \eqref{eq:17} and \eqref{eq:18} is negative. 
Since $a>3$, we have
\begin{equation*}
  \sup_{A\in\A_1}F(A)\leq C. 
\end{equation*}
For any sequence $\set{A_k}\subset\A_1$ with corresponding $r_k$ and
$a_k$, only one of the three cases
$r_k\to\infty$, $r_k\to0$, $a_k\to\infty$ can occur. And whenever it occurs, by \eqref{eq:11}, there must be
$a_k\to\infty$. Then $F(A_k)\to0$.

\textbf{Case $\A_2$.}
This case is simple. Since $1/3\leq a\leq3$, there exist positive constants
$C_n$ and $\tilde{C}_n$ such that for any $x\in\uS$, we have
\begin{equation*} 
    C_n r^{-1}\leq |A^{-1}x|\leq \tilde{C}_n r^{-1}.
\end{equation*}
Then
\begin{equation*}
  Cr^q \leq\int_{\uS} |A^{-1}x|^{-q} \leq \tilde{C} r^q,
\end{equation*}
which together with the definition of $\A$ implies that
\begin{equation*} 
  C\leq r\leq\tilde{C}.
\end{equation*}
Now, we see
\begin{equation*}
    C\leq |Ax|\leq \tilde{C}, \quad \forall x\in\uS.
\end{equation*}
Therefore,
\begin{equation*}
    F(A)\leq C \int_{\uS} |x'|^\alpha \,|x_n|^\beta \dd x = \tilde{C},
\end{equation*}
where the last equality is due to the fact that $\alpha>1-n$ and $\beta>-1$.
Obviously, for any sequence $\set{A_k}\subset\A_2$ with corresponding $r_k$ and
$a_k$, none of the four cases $r_k\to\infty$, $r_k\to0$, $a_k\to\infty$,
$a_k\to0$ can occur.

\textbf{Case $\A_3$.}
The discussion for this case is similar to that of $\A_1$.
We first claim that there exist two positive constants $C$ and $\tilde{C}$ such that for any
$A\in\A_3$, we have
\begin{equation} \label{eq:25}
  \begin{aligned} 
    Ca^{-\frac{1}{n}} &\leq r \leq \tilde{C}a^{-\frac{1}{n}},  && \text{ if } q<n-1, \\
    Ca^{-\frac{1}{n}}|\log a|^{-\frac{1}{n-1}} &\leq r
    \leq \tilde{C}a^{-\frac{1}{n}}|\log a|^{-\frac{1}{n-1}},  && \text{ if } q=n-1, \\
    Ca^{\frac{(n-1)(q-n)}{nq}} &\leq r \leq \tilde{C}a^{\frac{(n-1)(q-n)}{nq}},  && \text{ if } q>n-1.
  \end{aligned}
\end{equation}

In fact, from \eqref{eq:9} one has
\begin{equation} \label{eq:22}
    \int_{\uS} |A^{-1}x|^{-q}
    = 2r^q a^{\frac{q}{n}} \,\omega_{n-2} \int_0^{\frac{\pi}{2}}
    (\sin^2\theta+a^{2}\cos^2\theta)^{-\frac{q}{2}} \sin^{n-2}\theta \dd\theta.
\end{equation}
Since $a<1/3$, we have
\begin{equation} \label{eq:21}
  \begin{split}
    \int_{\uS} |A^{-1}x|^{-q}
    &\geq C r^q a^{\frac{q}{n}} \Bigl[  \int_0^{\frac{\pi}{3}} (\theta^2+a^2)^{-\frac{q}{2}} \theta^{n-2} \dd\theta 
    +\int_{\frac{\pi}{3}}^{\frac{\pi}{2}} \sin^{n-2}\theta \dd\theta \Bigr] \\
    &\geq C r^q a^{\frac{q}{n}} \Bigl[ \int_0^{\frac{\pi}{3}} (\theta^2+a^2)^{-\frac{q}{2}} \theta^{n-2} \dd\theta 
    +1 \Bigr].
  \end{split}
\end{equation}
Note that
\begin{equation*}
  \begin{split}
    \int_0^{\frac{\pi}{3}} (\theta^2+a^2)^{-\frac{q}{2}} \theta^{n-2} \dd\theta 
    &\geq \int_0^a 2^{-\frac{q}{2}} a^{-q} \theta^{n-2} \dd\theta 
    +\int_a^{\frac{\pi}{3}} 2^{-\frac{q}{2}} \theta^{-q} \theta^{n-2} \dd\theta \\
    &= \frac{C_q}{n-1} a^{-q+n-1} 
    +C_q \int_a^{\frac{\pi}{3}} \theta^{-q+n-2} \dd\theta \\
    &\geq C \begin{cases}
      1, & \text{ if } q<n-1, \\
      |\log a|, & \text{ if } q=n-1, \\
      a^{-q+n-1}, & \text{ if } q>n-1.
    \end{cases}
  \end{split}
\end{equation*}
Inserting it into \eqref{eq:21}, we obtain
\begin{equation} \label{eq:23}
  \int_{\uS} |A^{-1}x|^{-q}
  \geq C r^qa^{\frac{q}{n}}
  \begin{cases}
    1, & \text{ if } q<n-1, \\
    |\log a|, & \text{ if } q=n-1, \\
    a^{-q+n-1}, & \text{ if } q>n-1.
  \end{cases}
\end{equation}

On the other hand, from \eqref{eq:22}, we also have
\begin{equation*}
    \int_{\uS} |A^{-1}x|^{-q}
    \leq \tilde{C} r^q a^{\frac{q}{n}} \Bigl[  \int_0^{\frac{\pi}{3}} (\theta^2+a^2)^{-\frac{q}{2}} \theta^{n-2} \dd\theta 
    +\int_{\frac{\pi}{3}}^{\frac{\pi}{2}} \sin^{n-2}\theta \dd\theta \Bigr].
\end{equation*}
Similarly as in \eqref{eq:21}--\eqref{eq:23}, one can obtain
\begin{equation} \label{eq:24}
  \int_{\uS} |A^{-1}x|^{-q}
  \leq \tilde{C} r^qa^{\frac{q}{n}}
  \begin{cases}
    1, & \text{ if } q<n-1, \\
    |\log a|, & \text{ if } q=n-1, \\
    a^{-q+n-1}, & \text{ if } q>n-1.
  \end{cases}
\end{equation}

By combining \eqref{eq:23}, \eqref{eq:24} and the definition of $\A$ (see
\eqref{eq:5}), we have
\begin{equation*}
  \begin{aligned} 
    \tilde{C}^{-1}\Lambda^{-1} &\leq r^qa^{\frac{q}{n}} \leq C^{-1}\Lambda,  && \text{ if } q<n-1, \\
    \tilde{C}^{-1}\Lambda^{-1} &\leq r^{n-1}a^{(n-1)/n}|\log a| \leq C^{-1}\Lambda,  && \text{ if } q=n-1, \\
    \tilde{C}^{-1}\Lambda^{-1} &\leq r^qa^{(n-1)(n-q)/n} \leq C^{-1}\Lambda,  && \text{ if } q>n-1,
  \end{aligned}
\end{equation*}
which implies the claim \eqref{eq:25}.

Now, we can estimate $F(A)$ in \eqref{FA}.
From \eqref{eq:7}, we have
\begin{equation*}
    F(A)
    = 2r^p a^{\frac{p}{n}-p} \omega_{n-2} \int_0^{\frac{\pi}{2}} \sin^\alpha\!\theta \cos^\beta\!\theta
    \,(a^2\sin^2\theta+\cos^2\theta)^{\frac{p}{2}} \sin^{n-2}\theta \dd\theta.
\end{equation*}
Recalling that $a<1/3$, we can obtain
\begin{equation} \label{eq:19}
  \begin{split}
    F(A)
    &\leq C r^p a^{\frac{p}{n}-p} \Bigl[
    \int_0^{\frac{\pi}{6}} \theta^\alpha \theta^{n-2} \dd\theta 
    + \int_{\frac{\pi}{6}}^{\frac{\pi}{2}} (\tfrac{\pi}{2}-\theta)^\beta
    \bigl(a^2+(\tfrac{\pi}{2}-\theta)^2\bigr)^{\frac{p}{2}} \dd\theta \Bigr]\\
    &\leq C r^p a^{\frac{p}{n}-p} \Bigl[
    1
    + \int_0^{\frac{\pi}{3}} t^\beta \bigl(a^2+t^2\bigr)^{\frac{p}{2}} \dd t \Bigr].
  \end{split}
\end{equation}
And
\begin{equation*}
  \begin{split}
    \int_0^{\frac{\pi}{3}} t^\beta \bigl(a^2+t^2\bigr)^{\frac{p}{2}} \dd t
    &\leq \int_0^a t^\beta a^p \dd t +\int_a^{\frac{\pi}{3}} t^\beta t^p \dd t \\
    &= \frac{a^{p+\beta+1}}{\beta+1} +\int_a^{\frac{\pi}{3}} t^{\beta+p} \dd t \\
    &\leq C \begin{cases}
      1, & \text{ if } \beta+p+1>0, \\
      |\log a|, & \text{ if } \beta+p+1=0, \\
      a^{\beta+p+1}, & \text{ if } \beta+p+1<0.
    \end{cases}
  \end{split}
\end{equation*}
Inserting it into \eqref{eq:19}, we obtain
\begin{equation} \label{eq:20}
  F(A)\leq Cr^pa^{\frac{p}{n}-p}
  \begin{cases}
    1, & \text{ if } \beta+p+1>0, \\
    |\log a|, & \text{ if } \beta+p+1=0, \\
    a^{\beta+p+1}, & \text{ if } \beta+p+1<0.
  \end{cases}
\end{equation}

Recall that any $A\in\A_3$ must satisfy \eqref{eq:25}. The estimate given by
\eqref{eq:20} can be simplified as follows.
When $q<n-1$, there is
\begin{equation} \label{eq:26}
  F(A)\leq C
  \begin{cases}
    a^{-p}, & \text{ if } \beta+p+1>0, \\
    a^{-p}|\log a|, & \text{ if } \beta+p+1=0, \\
    a^{\beta+1}, & \text{ if } \beta+p+1<0.
  \end{cases}
\end{equation}
When $q=n-1$, there is
\begin{equation} \label{eq:27}
  F(A)\leq C
  \begin{cases}
    a^{-p}|\log a|^{-p/(n-1)}, & \text{ if } \beta+p+1>0, \\
    a^{-p}|\log a|^{(n-1-p)/(n-1)}, & \text{ if } \beta+p+1=0, \\
    a^{\beta+1}|\log a|^{-p/(n-1)}, & \text{ if } \beta+p+1<0.
  \end{cases}
\end{equation}
And when $q>n-1$, there is
\begin{equation} \label{eq:28}
  F(A)\leq C
  \begin{cases}
    a^{p(1-n)/q}, & \text{ if } \beta+p+1>0, \\
    a^{p(1-n)/q}|\log a|, & \text{ if } \beta+p+1=0, \\
    a^{p(1-n)/q+\beta+p+1}, & \text{ if } \beta+p+1<0.
  \end{cases}
\end{equation}

By the assumptions on $p$, $q$ and $\beta$, we see that the power of $a$ in
each case of \eqref{eq:26}, \eqref{eq:27} and \eqref{eq:28} is positive. Since
$a<1/3$, we have
\begin{equation*}
  \sup_{A\in\A_3}F(A)\leq C. 
\end{equation*}
For any sequence $\set{A_k}\subset\A_3$ with corresponding $r_k$ and
$a_k$, only one of the three cases
$r_k\to\infty$, $r_k\to0$, $a_k\to0$
can occur. And whenever it occurs, by \eqref{eq:25}, there must be
$a_k\to0$. Then $F(A_k)\to0$. 

Therefore, the proof of Lemma \ref{lem01} is completed. 
\end{proof}

\vskip5pt

\section{The non-uniqueness of solutions}\label{s4}

In this section, we prove Theorem \ref{thm1}.
Let $0<\epsilon<1/2$ and let
$M_\epsilon\in\mathrm{GL}(n)$ be given by
\begin{equation*}
  M_\epsilon =\DIAG(\epsilon, \cdots, \epsilon, 1) =
  \begin{pmatrix}
    \epsilon I & 0 \\
    0&1
  \end{pmatrix},
\end{equation*}
where $I$ is the unit $(n-1)\times(n-1)$ matrix.

For given indices $p, q$ such that $p<0<q$, we can choose appropriate indices $\alpha$ and $\beta$ such that they
are non-negative even integers satisfying the assumptions of Theorem \ref{thm2}. 
Let $\delta\in(0,-p)$ be a positive number (depending only on $p,q$, but independent of $\epsilon$) to be determined.

Now, consider the following equation
\begin{equation} \label{eq:1}
  \det(\nabla^2h+hI)(x) =|x'|^\alpha \,|x_n|^\beta \,|M_\epsilon x|^{-p-\delta-1-\beta},
  \quad x\in\uS.
\end{equation}
Note that this is an equation for the classical Minkowski problem. 
Since its right hand side is even
with respect to the origin and integrable on $\uS$, there exists a solution $h_\epsilon$, which is unique up to translation, \cite{CY}. 

By a translation, we may assume $h_\epsilon$ is the unique solution such that its associated convex body
$K_{h_\epsilon}$ centred at the origin.
Note that $K_{h_\epsilon}$ is rotationally symmetric and even.
The following lemma provides uniform bounds of $h_\epsilon$. 
The positive constants $C$, $\tilde{C}$ and $C_i$ in the following context depend only on $n$, $p$, $q$,
$\alpha$, $\beta$ and $\delta$, but independent of $\epsilon$.

\begin{lemma}\label{lem03}
  There exists a positive constant $C$, independent of $\epsilon\in(0,1/2)$, such
  that
  \begin{equation} \label{eq:2}
    C^{-1}\leq h_\epsilon \leq C \quad \text{ on }\uS.
  \end{equation}
\end{lemma}

\begin{proof}
  We first claim that the area of $\pd K_{h_\epsilon}$ is uniformly bounded from
  above. In fact, by \eqref{eq:1}, one has
    \begin{equation*}
    \begin{split}
      \AREA(\pd K_{h_\epsilon})
      &= \int_{\uS} \det(\nabla^2h_\epsilon+h_\epsilon I)(x) \dd x \\
      &= \int_{\uS} |x'|^\alpha \,|x_n|^\beta \sqrt{\epsilon^2|x'|^2+x_n^2}^{-p-\delta-1-\beta} \dd x \\
      &\leq \int_{\uS} \sqrt{\epsilon^2|x'|^2+x_n^2}^{-p-\delta-1} \dd x.
    \end{split}
  \end{equation*}
Since $0<\delta<-p$, we have
  \begin{equation*}
    \begin{split}
      \AREA(\pd K_{h_\epsilon})
      &\leq \begin{cases}
        \int_{\uS} |x_n|^{-p-\delta-1} \dd x & \text{when } 0<-p-\delta<1 \\[1.3ex]
        \int_{\uS} \dd x & \text{when } 1\leq -p-\delta
      \end{cases} \\
      &\leq C,
    \end{split}
  \end{equation*}
where $C$ is a positive constant depending on $n$, $p$, $\delta$, but independent
of $\epsilon$.
By the isoperimetric inequality, we also obtain
\begin{equation} \label{eq:3}
  \VOL(K_{h_\epsilon})\leq C_n\AREA(\pd K_{h_\epsilon})^{\frac{n}{n-1}} \leq C.
\end{equation}

Let $E_{\epsilon}$ be the minimum ellipsoid of
$K_{h_\epsilon}$. Then,
\begin{equation*}
\frac{1}{n} E_{\epsilon} \subset K_{h_\epsilon} \subset E_{\epsilon},
\end{equation*}
which implies that
\begin{equation} \label{eq:32}
\frac{1}{n} h_{E_\epsilon} \leq h_\epsilon \leq h_{E_\epsilon} \quad \text{on } \uS,
\end{equation}
where $h_{E_\epsilon}$ is the support function of $E_\epsilon$.
Since $K_{h_\epsilon}$ is rotationally symmetric and even, $E_{\epsilon}$ is
also rotationally symmetric and even. In particular, the centre of
$E_{\epsilon}$ is at the origin. Let $R_{1\epsilon}, \cdots, R_{n\epsilon}$ be
the lengths of the semi-axes of $E_\epsilon$ along the $x_1, \cdots, x_n$ axes.
Then $R_{1\epsilon}=\cdots=R_{n-1;\epsilon}$, and
\begin{equation}\label{eq:34}
h_{E_\epsilon}(x) =\sqrt{R_{1\epsilon}^2|x'|^2+R_{n\epsilon}^2x_n^2}, \qquad \forall\, x\in\uS.
\end{equation}

From the equation \eqref{eq:1}, one can see that
\begin{equation*}
  \begin{split}
    \VOL(K_{h_\epsilon})
    &= \frac{1}{n}\int_{\uS} h_\epsilon \det(\nabla^2h_\epsilon+h_\epsilon I) \\
    &= \frac{1}{n}\int_{\uS} h_\epsilon(x)
    |x'|^\alpha \,|x_n|^\beta \sqrt{\epsilon^2|x'|^2+x_n^2}^{-p-\delta-1-\beta} \dd x \\
    &\geq \frac{1}{n}\int_{\uS} h_\epsilon(x)
    |x'|^\alpha \,|x_n|^\beta \sqrt{\epsilon^2|x'|^2+x_n^2}^{-p-\delta-1} \dd x \\
    &\geq \frac{1}{n}
    \begin{cases}
      \int_{\uS} h_\epsilon(x)
      |x'|^\alpha \,|x_n|^\beta \dd x
      & \text{when } 0<-p-\delta<1, \\[1.3ex]
      \int_{\uS} h_\epsilon(x)
      |x'|^\alpha \,|x_n|^{\beta-p-\delta-1} \dd x
      & \text{when } 1\leq -p-\delta.
    \end{cases}
  \end{split}
\end{equation*}
From \eqref{eq:32} and \eqref{eq:34},  
\begin{equation*}
h_\epsilon(x) \geq \frac{1}{\sqrt{2}\,n} (R_{1\epsilon}|x'|+R_{n\epsilon}|x_n|), \qquad \forall\, x\in\uS.
\end{equation*}
Thus, we obtain
\begin{equation*}
  \begin{split}
    \VOL(K_{h_\epsilon})
    &\geq 
    \begin{cases}
      \displaystyle\int_{\uS} 
      \frac{R_{1\epsilon}|x'|+R_{n\epsilon}|x_n|}{\sqrt{2}n^2}
      |x'|^\alpha \,|x_n|^\beta \dd x
      & \text{when } 0<-p-\delta<1 \\[2.6ex]
      \displaystyle\int_{\uS} 
      \frac{R_{1\epsilon}|x'|+R_{n\epsilon}|x_n|}{\sqrt{2}n^2}
      |x'|^\alpha \,|x_n|^{\beta-p-\delta-1} \dd x
      & \text{when } 1\leq -p-\delta
    \end{cases} \\
    &=C_1R_{1\epsilon}+C_2R_{n\epsilon} \\
    &\geq C(R_{1\epsilon}+R_{n\epsilon}),
  \end{split}
\end{equation*}
where $C_1$, $C_2$ and $C$ are positive constants depending on $n$, $p$,
$\alpha$, $\beta$, $\delta$, but independent of $\epsilon$.
Therefore, again by \eqref{eq:32} we have
\begin{equation} \label{eq:35}
  \begin{split}
    \max h_\epsilon
    &\leq \max h_{E_\epsilon} \\
    &< R_{1\epsilon}+R_{n\epsilon} \\
    &\leq C \VOL(K_{h_\epsilon}) \\
    &\leq C,
  \end{split}
\end{equation}
where the last inequality is due to \eqref{eq:3}.
The second inequality of \eqref{eq:2} is proved.

It remains to prove the first inequality of \eqref{eq:2}.
For convenience, write
\begin{align*}
r_\epsilon &:=\min\set{R_{1\epsilon}, R_{n\epsilon}}, \\
R_\epsilon &:=\max\set{R_{1\epsilon}, R_{n\epsilon}}.
\end{align*}
Recalling \eqref{eq:32}, one has
\begin{align*}
r_\epsilon &\leq n \min h_\epsilon, \\
R_\epsilon &\leq n \max h_\epsilon.
\end{align*}
We can estimate $\VOL(K_{h_\epsilon})$ as follows:
\begin{equation*}
  \begin{split}
    \VOL(K_{h_\epsilon})
    &\leq \VOL(E_\epsilon) \\
    &= \kappa_n R_{1\epsilon}^{n-1}R_{n\epsilon} \\
    &\leq \kappa_n R_\epsilon^{n-2} R_{1\epsilon}R_{n\epsilon} \\ 
    &= \kappa_n R_\epsilon^{n-1} r_\epsilon \\ 
    &\leq C_n (\max h_\epsilon)^{n-1} \cdot \min h_\epsilon,
  \end{split}
\end{equation*}
where $\kappa_n$ is the volume of the unit ball in $\R^n$.
By the third inequality of \eqref{eq:35}, namely $\max h_\epsilon \leq
C\VOL(K_{h_\epsilon})$, the above inequality reads
\begin{equation*}
\max h_\epsilon \leq C (\max h_\epsilon)^{n-1} \cdot \min h_\epsilon,
\end{equation*}
namely
\begin{equation*}
1 \leq C (\max h_\epsilon)^{n-2} \cdot \min h_\epsilon.
\end{equation*}
Now the first inequality of \eqref{eq:2} follows from its second inequality.
The proof of this lemma is completed.
\end{proof}

Define
\begin{equation} \label{H_eps}
  H_\epsilon(x) := \epsilon^{\frac{q+\delta-1}{q-p}}
  |M_\epsilon^{-1} x| \cdot h_\epsilon\biggl( \frac{M_\epsilon^{-1} x}{|M_\epsilon^{-1} x|} \biggr),
  \quad x\in\uS.
\end{equation}

\begin{lemma} \label{lem02}
  The function $H_\epsilon$ satisfies the equation
  \begin{equation} \label{eq:29}
  H_\epsilon^{1-p} |\delbar H_\epsilon|^{q-n} \det(\nabla^2H_\epsilon+H_\epsilon I) = f_\epsilon \quad \text{ on } \ \uS,
  \end{equation}
  where
  \begin{equation} \label{eq:30}
    f_\epsilon(x) :=
    h_\epsilon(x_\epsilon)^{1-p}
    |x'|^\alpha \,|x_n|^\beta \,|N_\epsilon x|^{\delta-\alpha-n+1} 
    |N_\epsilon(\delbar h_\epsilon)(x_\epsilon)|^{q-n},
  \end{equation}
  and for simplicity, we also denote
  \begin{equation} \label{eq:31}
    x_\epsilon = \frac{M_\epsilon^{-1}x}{|M_\epsilon^{-1}x|}\quad
    \text{ and }\quad
    N_\epsilon
    =\epsilon M_\epsilon^{-1}
    =
  \begin{pmatrix}
    I & 0 \\
    0&\epsilon
  \end{pmatrix}.
\end{equation}
\end{lemma}

\begin{proof}
Let
\begin{equation} \label{u_eps}
u_\epsilon(x) :=  |M_\epsilon^{-1} x| \cdot h_\epsilon\biggl( \frac{M_\epsilon^{-1} x}{|M_\epsilon^{-1} x|} \biggr).
\end{equation}
By the invariance of the quantity
$h_\epsilon^{n+1}\det(\nabla^2h_\epsilon+h_\epsilon I)$ under linear
transformations, see Proposition 7.1 in \cite{CW.Adv.205-2006.33}
or formula (2.12) in \cite{LW.JDE.254-2013.983}, we have
\begin{equation} \label{eq:33}
    \det(\nabla^2u_\epsilon+u_\epsilon I)(x)
    = \det(\nabla^2h_\epsilon+h_\epsilon I)\biggl( \frac{M_\epsilon^{-1} x}{|M_\epsilon^{-1} x|} \biggr)
    \cdot \frac{(\det M_\epsilon^{-1})^2}{|M_\epsilon^{-1}x|^{n+1}}.
\end{equation}

Observe that
\begin{equation*}
  x_\epsilon
  = \frac{M_\epsilon^{-1}x}{|M_\epsilon^{-1}x|} 
  = \frac{(\epsilon^{-1}x',x_n)}{|M_\epsilon^{-1}x|} 
  = \frac{(x',\epsilon x_n)}{|N_\epsilon x|}.
\end{equation*}
By virtue of \eqref{eq:1}, we then have
\begin{equation*}
  \begin{split}
    \det(\nabla^2h_\epsilon+h_\epsilon I)\biggl( \frac{M_\epsilon^{-1} x}{|M_\epsilon^{-1} x|} \biggr)
    &= 
    \frac{|x'|^\alpha}{|N_\epsilon x|^\alpha} \cdot
    \frac{|\epsilon x_n|^\beta}{|N_\epsilon x|^\beta} \cdot
    \biggl( \frac{1}{|M_\epsilon^{-1} x|} \biggr)^{-p-\delta-1-\beta} \\
    &= 
    \frac{|x'|^\alpha}{|N_\epsilon x|^\alpha} \cdot
    \frac{|\epsilon x_n|^\beta}{|N_\epsilon x|^\beta} \cdot
    \biggl( \frac{\epsilon}{|N_\epsilon x|} \biggr)^{-p-\delta-1-\beta} \\
    &= \epsilon^{-p-\delta-1}
    |x'|^\alpha \,|x_n|^\beta \,|N_\epsilon x|^{p+\delta+1-\alpha}.
  \end{split}
\end{equation*}

Inserting it into \eqref{eq:33}, we obtain
\begin{equation} \label{eq:36}
  \begin{split}
    \det(\nabla^2u_\epsilon+u_\epsilon I)(x)
    &= \epsilon^{-p-\delta-1}
    |x'|^\alpha \,|x_n|^\beta \,|N_\epsilon x|^{p+\delta+1-\alpha}
    \cdot \frac{(\epsilon^{1-n})^2 \,\epsilon^{n+1}}{|N_\epsilon x|^{n+1}} \\
    &= \epsilon^{-p-\delta+2-n}
    |x'|^\alpha \,|x_n|^\beta \,|N_\epsilon x|^{p+\delta-\alpha-n}.
  \end{split}
\end{equation}

By the definition of $u_\epsilon$ in \eqref{u_eps}, one can see that
\begin{gather*}
  u_\epsilon(x)
  =\epsilon^{-1} |N_\epsilon x| \cdot h_\epsilon(x_\epsilon), \\
  (\delbar u_\epsilon)(x)
  =M_\epsilon^{-T}(\delbar h_\epsilon)(x_\epsilon) 
  =\epsilon^{-1} N_\epsilon(\delbar h_\epsilon)(x_\epsilon).
\end{gather*}
Hence, from \eqref{eq:36} we have
\begin{equation*} 
  \begin{split}
    u_\epsilon^{1-p} |\delbar u_\epsilon|^{q-n}
    & \det(\nabla^2u_\epsilon+u_\epsilon I)(x) \\
    &=
    \epsilon^{p-1} |N_\epsilon x|^{1-p} h_\epsilon(x_\epsilon)^{1-p}\cdot 
    \epsilon^{n-q} |N_\epsilon(\delbar h_\epsilon)(x_\epsilon)|^{q-n} \cdot\\
    &\hskip1.3em
    \epsilon^{-p-\delta+2-n}
    |x'|^\alpha \,|x_n|^\beta \,|N_\epsilon x|^{p+\delta-\alpha-n} \\
    &=
    \epsilon^{-q-\delta+1} h_\epsilon(x_\epsilon)^{1-p}
    |x'|^\alpha \,|x_n|^\beta \,|N_\epsilon x|^{\delta-\alpha-n+1} 
    |N_\epsilon(\delbar h_\epsilon)(x_\epsilon)|^{q-n}.
  \end{split}
\end{equation*}
Recalling \eqref{H_eps}, namely $H_\epsilon = \epsilon^{\frac{q+\delta-1}{q-p}} u_\epsilon$, we thus obtain
\begin{multline*}
  H_\epsilon^{1-p} |\delbar H_\epsilon|^{q-n} \det(\nabla^2H_\epsilon+H_\epsilon I)(x) \\
  =h_\epsilon(x_\epsilon)^{1-p}
  |x'|^\alpha \,|x_n|^\beta \,|N_\epsilon x|^{\delta-\alpha-n+1} 
  |N_\epsilon(\delbar h_\epsilon)(x_\epsilon)|^{q-n},
\end{multline*}
that is equation \eqref{eq:29}. The proof is done.
\end{proof}

From the definition of $H_\epsilon$ in \eqref{H_eps}, one can see that
\begin{equation*}
  \rho_{H_\epsilon}(u)
  = \epsilon^{\frac{q+\delta-1}{q-p}} 
  |M_\epsilon u|^{-1} \cdot \rho_{h_\epsilon}\biggl( \frac{M_\epsilon u}{|M_\epsilon u|} \biggr),
  \quad \forall\, u\in\uS.
\end{equation*}
By Lemma \ref{lem03}, one has
\begin{equation}\label{newrrr}
  C^{-1}\leq \rho_{h_\epsilon}\leq C\quad \text{ on }\ \uS,
\end{equation}
where $C>0$ is independent of $\epsilon$.
Therefore, we have
\begin{equation}\label{newrho}
  \rho_{H_\epsilon}(u)
  \leq C\, \epsilon^{\frac{q+\delta-1}{q-p}}
  |M_\epsilon u|^{-1}
  \quad \forall\, u\in\uS.
\end{equation}

\begin{corollary}
We have the estimate for the $q$-th dual volume of $K_{H_\epsilon}$ as follows: 
\begin{equation} \label{eq:38}
  \widetilde{V}_q(K_{H_\epsilon})
  \leq C
  \begin{cases}
    \epsilon^{\frac{q(q+\delta-1)}{q-p}}, & \text{ if } q<1, \\
    \epsilon^{\frac{\delta}{1-p}}|\log\epsilon|, & \text{ if } q=1, \\
    \epsilon^{\frac{q\delta+p(q-1)}{q-p}}, & \text{ if } q>1.
  \end{cases}
\end{equation}
\end{corollary}

\begin{proof}
From \eqref{newrho} and the definition \eqref{qdvol},
\begin{equation}\label{eq:37}
  \begin{split}
    \widetilde{V}_q(K_{H_\epsilon})
    &= \frac{1}{n} \int_{\uS} \rho_{H_\epsilon}^q(u) \dd u \\
    &\leq C\, \epsilon^{\frac{q(q+\delta-1)}{q-p}} \int_{\uS} |M_\epsilon u|^{-q} \dd u.
  \end{split}
\end{equation}
Note that
\begin{equation*}
  \begin{split}
    \int_{\uS} |M_\epsilon u|^{-q} \dd u
    &= \int_{\uS} (\epsilon^2|u'|^2+u_n^2)^{-\frac{q}{2}} \dd u \\
    &=  2\omega_{n-2} \int_0^{\frac{\pi}{2}}
    (\epsilon^2\sin^2\theta+\cos^2\theta)^{-\frac{q}{2}} \sin^{n-2}\theta \dd\theta \\
    &\leq C  \int_0^{\frac{\pi}{2}}
    (\epsilon\sin\theta+\cos\theta)^{-q} \dd\theta,
  \end{split}
\end{equation*}
where $C$ is a positive constant independent of $\epsilon\in(0,1/2)$.
Since
\begin{equation*}
  \begin{split}
    \int_0^{\frac{\pi}{2}}
    (\epsilon\sin\theta+\cos\theta)^{-q} \dd\theta
    &\leq \int_0^{\frac{\pi}{4}} (\cos\theta)^{-q} \dd\theta 
    +C_q \int_{\frac{\pi}{4}}^{\frac{\pi}{2}} (\epsilon+\tfrac{\pi}{2}-\theta)^{-q} \dd\theta \\
    &= \tilde{C}_q 
    +C_q \int_0^{\frac{\pi}{4}} (\epsilon+t)^{-q} \dd t \\
    &\leq C_q
    \begin{cases}
      1, & \text{ if } q<1, \\
      |\log\epsilon|, & \text{ if } q=1, \\
      \epsilon^{1-q}, & \text{ if } q>1,
    \end{cases}
  \end{split}
\end{equation*}
we obtain that
\begin{equation*}
    \int_{\uS} |M_\epsilon u|^{-q} \dd u
\leq C
    \begin{cases}
      1, & \text{ if } q<1, \\
      |\log\epsilon|, & \text{ if } q=1, \\
      \epsilon^{1-q}, & \text{ if } q>1.
    \end{cases}
\end{equation*}
Therefore, the estimate \eqref{eq:38} follows. 
\end{proof}

On the other hand, thanks to \eqref{eq:2} and \eqref{newrrr}, the function $f_\epsilon$ in \eqref{eq:30} satisfies the assumption of Theorem \ref{thm2}.
By applying Theorem \ref{thm2} to equation \eqref{eq:29}, we can obtain another solution, say $\widetilde{H}_\epsilon$.
In order to show $\widetilde{H}_\epsilon$ is different to $H_\epsilon$, we shall show that they have different $q$-th dual volumes. 
By the estimate \eqref{var-est}, the $q$-th dual volume $\widetilde{V}_q(K_{\widetilde{H}_\epsilon})$ is bounded below by $\norm{f_\epsilon}_{L^1(\uS)}$.
In the following we shall estimate $\norm{f_\epsilon}_{L^1(\uS)}$.

For simplicity, we write
\begin{equation}\label{eq:43}
  \delbar h_\epsilon =(\xi_\epsilon',\xi_{\epsilon n}), 
\end{equation}
where $\xi_\epsilon'$ denotes the first $n-1$ coordinates of
$\delbar h_\epsilon$, and $\xi_{\epsilon n}$ the last one.
Let
\begin{equation}\label{eq:41}
  S_0 :=\set{x\in \uS : |x_n|<\cos\Bigl( \frac{\pi}{4}+\frac{1}{2}\arccos C^{-2} \Bigr)},
\end{equation}
where $C$ is the constant in Lemma \ref{lem03}.

\begin{lemma} \label{lem05}
  For any $\epsilon\in(0,1/2)$, one has
  \begin{equation} \label{eq:39}
    |\xi_\epsilon'(x)|> \frac{1}{2}C^{-1}, \quad \forall\, x\in S_0,
  \end{equation}
where $C$ is the same constant as in Lemma \ref{lem03} and \eqref{eq:41}.
\end{lemma}

\begin{proof}
  By Lemma \ref{lem03}, one has
  \begin{equation*}
B_{C^{-1}}\subset K_{h_\epsilon} \subset B_C,
\end{equation*}
where $B_r$ denotes the ball in $\R^n$ centred at the origin with radius $r$.

For a fixed $x\in S_0$, $(\delbar h_\epsilon)(x)$ is the boundary point of
$K_{h_\epsilon}$, whose unit outer normal vector is $x$. Thus
\begin{equation*}
  \left\langle (\delbar h_\epsilon)(x), x \right\rangle \geq C^{-1}
  \quad\text{ and }\quad
  |(\delbar h_\epsilon)(x)|\leq C,
\end{equation*}
where $\langle \cdot, \cdot \rangle$ is the inner product of $\R^n$.
We shall show that the estimate \eqref{eq:39} follows.

In fact, using $\delbar h_\epsilon =(\xi_\epsilon',\xi_{\epsilon n})$ and
$x=(x',x_n)$, we have
\begin{equation*}
  \begin{split}
    C^{-1}
    &\leq x'\cdot\xi_\epsilon'+x_n\cdot\xi_{\epsilon n} \\
    &\leq |x'|\cdot|\xi_\epsilon'|+|x_n|\cdot|\xi_{\epsilon n}| \\
    &\leq |x'|\cdot|\xi_\epsilon'|+|x_n|\cdot\sqrt{C^2-|\xi_\epsilon'|^2}.
  \end{split}
\end{equation*}
Write $x_n=\cos\theta$, where $\theta\in[0,\pi]$, and
\begin{equation} \label{eq:40}
 |\xi_\epsilon'| =C \sin\phi, \quad \phi\in[0,\tfrac{\pi}{2}].
\end{equation}
Then
\begin{equation*}
  \begin{split}
    C^{-1}
    &\leq C\sin\theta\sin\phi+C\cos\theta\cos\phi \\
    &=C\cos(\phi-\theta),
  \end{split}
\end{equation*}
which implies that
\begin{equation*}
|\phi-\theta|\leq\arccos C^{-2}.
\end{equation*}
Therefore, we have
\begin{equation}\label{eq:42}
\phi\geq\theta-\arccos C^{-2}.
\end{equation}
Recalling the definition of $S_0$ in \eqref{eq:41}, one has
\begin{equation*}
  %-\cos\Bigl( \frac{\pi}{4}+\frac{1}{2}\arccos C^{-2} \Bigr)<
  \cos\theta <\cos\Bigl( \frac{\pi}{4}+\frac{1}{2}\arccos C^{-2} \Bigr),
\end{equation*}
which implies
\begin{equation*}
  \theta > \frac{\pi}{4}+\frac{1}{2}\arccos C^{-2}.
\end{equation*}
Hence, we obtain from \eqref{eq:42} that
\begin{equation*}
\phi>\frac{\pi}{4}-\frac{1}{2}\arccos C^{-2}.
\end{equation*}
By the monotonicity and concavity of $\sin$ on $[0,\tfrac{\pi}{2}]$, one has
\begin{equation*}
  \begin{split}
    \sin\phi
    &>\sin\Bigl( \frac{\pi}{4}-\frac{1}{2}\arccos C^{-2} \Bigr) \\
    &>\frac{1}{2} \sin\Bigl( \frac{\pi}{2}-\arccos C^{-2} \Bigr) \\
    &=\frac{1}{2} C^{-2},
  \end{split}
\end{equation*}
which together with \eqref{eq:40} implies that
$|\xi_\epsilon'|>\frac{1}{2}C^{-1}$. The proof is finished. 
\end{proof}

Using Lemma \ref{lem05}, we can show that $\norm{f_\epsilon}_{L^1(\uS)}$ has a
uniform positive lower bound for all $\epsilon\in(0,1/2)$.

\begin{lemma}\label{lem04}
  For $f_\epsilon$ given in \eqref{eq:30}, we have
  \begin{equation*}
\norm{f_\epsilon}_{L^1(\uS)} \geq C,
  \end{equation*}
 where the constant $C>0$ is independent of $\epsilon$. 
\end{lemma}

\begin{proof}{}
  We start from the following inequality:
  \begin{equation*}
    \begin{split}
      f_\epsilon(x)
      &=
      h_\epsilon(x_\epsilon)^{1-p}
      |x'|^\alpha \,|x_n|^\beta \,|N_\epsilon x|^{\delta-\alpha-n+1} 
      |N_\epsilon(\delbar h_\epsilon)(x_\epsilon)|^{q-n} \\
      &\geq C
      |x'|^\alpha \,|x_n|^\beta \,|N_\epsilon x|^{\delta-\alpha-n+1} 
      |N_\epsilon(\delbar h_\epsilon)(x_\epsilon)|^{q-n},
    \end{split}
  \end{equation*}
  where Lemma \ref{lem03} is applied.
  Then
  \begin{equation*}
    \norm{f_\epsilon}_{L^1}
    \geq C \int_{\uS} 
    |x'|^\alpha \,|x_n|^\beta \,|N_\epsilon x|^{\delta-\alpha-n+1} 
    |N_\epsilon(\delbar h_\epsilon)(x_\epsilon)|^{q-n} \dd x.
  \end{equation*}
  
  We shall apply the integration by substitution 
  \begin{equation*}
    x
    = \frac{M_\epsilon y}{|M_\epsilon y|}
    = \frac{(\epsilon y',y_n)}{|M_\epsilon y|}
    \quad \text{ for } y\in\uS.
  \end{equation*}
From the notations in \eqref{eq:31}, we have 
  \begin{gather*}
    N_\epsilon x
    = \frac{\epsilon y}{|M_\epsilon y|}, \\ 
    x_\epsilon
    = \frac{M_\epsilon^{-1}x}{|M_\epsilon^{-1}x|} 
    =y.
  \end{gather*}
Hence, using the computation as in \cite[Lemma 2.2]{Lu.SCM.61-2018.511} one has
  \begin{equation}\label{eq:44}
    \begin{split}
      \norm{f_\epsilon}_{L^1}
      &\geq C \int_{\uS} 
      \frac{|\epsilon y'|^\alpha}{|M_\epsilon y|^\alpha} \cdot
      \frac{|y_n|^\beta}{|M_\epsilon y|^\beta}
      \biggl( \frac{\epsilon}{|M_\epsilon y|} \biggr)^{\delta-\alpha-n+1} 
      |N_\epsilon(\delbar h_\epsilon)(y)|^{q-n} \cdot
      \frac{\epsilon^{n-1}\dd y}{|M_\epsilon y|^n} \\
      &= C \epsilon^{\delta} \int_{\uS} 
      | y'|^\alpha
      \,|y_n|^\beta
      \,|M_\epsilon y|^{-\beta-\delta-1} 
      \,|N_\epsilon(\delbar h_\epsilon)(y)|^{q-n}
      \dd y \\
      &\geq C \epsilon^{\delta} \int_{S_0} 
      | y'|^\alpha
      \,|y_n|^\beta
      \,|M_\epsilon y|^{-\beta-\delta-1} 
      \,|N_\epsilon(\delbar h_\epsilon)(y)|^{q-n}
      \dd y.
    \end{split}
  \end{equation}
Again, by \eqref{eq:31} and
  the notation $\delbar h_\epsilon$ in \eqref{eq:43},
  we have
  \begin{equation*}
    \begin{split}
      |N_\epsilon(\delbar h_\epsilon)(y)|
      &= \sqrt{|\xi_\epsilon'(y)|^2+\epsilon^2|\xi_{\epsilon n}(y)|^2} \\
      &\geq |\xi_\epsilon'(y)| \\
      &>\frac{1}{2}C^{-1}, \quad \forall\, y\in S_0,
    \end{split}
  \end{equation*}
  where the last inequality is due to Lemma \ref{lem05}.
Observe that
  \begin{equation*}
    |N_\epsilon(\delbar h_\epsilon)(y)|
    \leq |\delbar h_\epsilon(y)|
    \leq C, \quad \forall\, y\in\uS,
  \end{equation*}
  where the last inequality is due to Lemma \ref{lem03}.
  Hence for any $\epsilon\in(0,1/2)$ we have
  \begin{equation*}
    \frac{1}{2}C^{-1}<|N_\epsilon(\delbar h_\epsilon)(y)|
    \leq C, \quad \forall\, y\in S_0. 
  \end{equation*}
  Inserting it into \eqref{eq:44}, one has that
  \begin{equation}\label{eq:45}
    \begin{split}
      \norm{f_\epsilon}_{L^1}
      &\geq C \epsilon^{\delta} \int_{S_0} 
      |y'|^\alpha
      \,|y_n|^\beta
      \,|M_\epsilon y|^{-\beta-\delta-1} 
      \dd y \\
      &\geq C \epsilon^{\delta} \int_{S_0} 
      |y_n|^\beta
      \,|M_\epsilon y|^{-\beta-\delta-1} 
      \dd y \\
      &= C \epsilon^{\delta} \int_{S_0} 
      |y_n|^\beta
      \sqrt{\epsilon^2|y'|^2+y_n^2}^{-\beta-\delta-1} 
      \dd y.
    \end{split}
  \end{equation}
  
  Recalling the definition of $S_0$ in \eqref{eq:41} and denoting
  \begin{equation*}
    \theta_0 := \frac{\pi}{4}+\frac{1}{2}\arccos C^{-2},
  \end{equation*}
  we have
  \begin{equation}\label{eq:47}
    \begin{split}
      \int_{S_0} |y_n|^\beta
      &\sqrt{\epsilon^2|y'|^2+y_n^2}^{-\beta-\delta-1} \dd y \\
      &= 2\omega_{n-2} \int_{\theta_0}^{\frac{\pi}{2}}
      \cos^\beta\!\theta
      \,(\epsilon^{2}\sin^2\theta+\cos^2\theta)^{-\frac{\beta+\delta+1}{2}} \sin^{n-2}\theta \dd\theta \\
      &\geq C \int_{\theta_0}^{\frac{\pi}{2}}
      \cos^\beta\!\theta
      \,(\epsilon\sin\theta+\cos\theta)^{-\beta-\delta-1} \dd\theta \\
      &= C \int_0^{\frac{\pi}{2}-\theta_0}
      \sin^\beta\!\phi
      \,(\epsilon\cos\phi+\sin\phi)^{-\beta-\delta-1} \dd\phi \\
      &\geq C \int_0^{\frac{\pi}{2}-\theta_0}
      \phi^\beta
      \,(\epsilon+\phi)^{-\beta-\delta-1} \dd\phi.
    \end{split}
  \end{equation} 
  When $0<\epsilon\leq \frac{\pi}{2}-\theta_0$, one has
  \begin{equation*}
    \begin{split}
      \int_0^{\frac{\pi}{2}-\theta_0}
      \phi^\beta
      \,(\epsilon+\phi)^{-\beta-\delta-1} \dd\phi
      &\geq
      \int_0^{\epsilon}
      \phi^\beta
      \,(\epsilon+\phi)^{-\beta-\delta-1} \dd\phi \\ 
      &\geq
      \int_0^{\epsilon}
      \phi^\beta
      \,(2\epsilon)^{-\beta-\delta-1} \dd\phi \\ 
      &= C \epsilon^{-\delta}.
    \end{split}
  \end{equation*}  
  When $\epsilon> \frac{\pi}{2}-\theta_0$, one has
  \begin{equation*}
    \begin{split}
      \int_0^{\frac{\pi}{2}-\theta_0}
      \phi^\beta
      \,(\epsilon+\phi)^{-\beta-\delta-1} \dd\phi
      &\geq
      \int_0^{\frac{\pi}{2}-\theta_0}
      \phi^\beta
      \,2^{-\beta-\delta-1} \dd\phi  \\
      &\geq C.
    \end{split}
  \end{equation*}

Therefore, for any $\epsilon\in(0,1/2)$,
  \begin{equation} \label{eq:46}
    \int_0^{\frac{\pi}{2}-\theta_0}
    \phi^\beta
    \,(\epsilon+\phi)^{-\beta-\delta-1} \dd\phi
    \geq
    C \epsilon^{-\delta}. 
  \end{equation}
Combining \eqref{eq:45}, \eqref{eq:47} and \eqref{eq:46}, we obtain that
  \begin{equation*}
    \norm{f_\epsilon}_{L^1}\geq C.
  \end{equation*}
The proof of Lemma \ref{lem04} is finished. 
\end{proof}

Now, we are ready to complete the proof of Theorem \ref{thm1} by comparing the $q$-th dual volumes of $K_{H_\epsilon}$ and $K_{\widetilde{H}_\epsilon}$.
\begin{proof}[\textbf{Proof of Theorem \ref{thm1}}]
  By Theorem \ref{thm2} and Lemma \ref{lem04}, we have
  \begin{equation}
\widetilde{V}_q(K_{\widetilde{H}_\epsilon}) \geq C>0,
\end{equation}
where $C$ depends only on $n$, $p$, $q$,
$\alpha$, $\beta$ and $\delta$, and is independent of $\epsilon$.

For any given $q\in(0,1)$ and $p\in(-\infty,q-1)$, we can choose $\alpha=0$,
$\beta=0$, and $\delta\in(1-q,-p)$. Then, by virtue of \eqref{eq:38}
\begin{equation*} 
  \widetilde{V}_q(K_{H_\epsilon})
  \leq C
  \epsilon^{\frac{q(q+\delta-1)}{q-p}}
  \to 0, \qquad\text{ as }\ \epsilon\to0^+,
\end{equation*}
which implies that $K_{H_\epsilon}$ and $K_{\widetilde{H}_\epsilon}$ are
different convex bodies, namely $H_\epsilon$ and $\widetilde{H}_\epsilon$ are
different solutions to equation \eqref{eq:29}.

For $q=1$ and $p\in(-\infty,0)$, we can choose $\alpha=0$,
$\beta=0$, and $\delta\in(0,-p)$. Then, by virtue of \eqref{eq:38}
\begin{equation*} 
  \widetilde{V}_q(K_{H_\epsilon})
  \leq C
  \epsilon^{\frac{\delta}{1-p}}|\log\epsilon|
  \to 0,\qquad \text{ as } \epsilon\to0^+,
\end{equation*}
which implies that $K_{H_\epsilon}$ and $K_{\widetilde{H}_\epsilon}$ are
different convex bodies, namely $H_\epsilon$ and $\widetilde{H}_\epsilon$ are
different solutions to equation \eqref{eq:29}.

For any given $q\in(1,+\infty)$ and $p\in(-\infty,0)$, we can choose $\alpha$
and $\beta$ to be non-negative even integers, and choose $\delta$ as
\begin{equation*}
-p\cdot\frac{q-1}{q}<\delta<-p.
\end{equation*}
Then, by virtue of \eqref{eq:38}
\begin{equation*} 
  \widetilde{V}_q(K_{H_\epsilon})
  \leq C
  \epsilon^{\frac{q\delta+p(q-1)}{q-p}}
  \to 0,\qquad \text{ as } \epsilon\to0^+,
\end{equation*}
which implies that $K_{H_\epsilon}$ and $K_{\widetilde{H}_\epsilon}$ are
different convex bodies, namely $H_\epsilon$ and $\widetilde{H}_\epsilon$ are
different solutions to equation \eqref{eq:29}.
  If, in addition, $1+\frac{n}{p}<\frac{1}{p}+\frac{1}{q}<\frac{n}{q}-1$ holds, $\alpha$ and $\beta$ can be
chosen as zero. Then $f_\epsilon$ in \eqref{eq:30} is positive on $\uS$.

Therefore, the proof of Theorem \ref{thm1} is completed.
\end{proof}

%\newpage

\end{document}